\documentclass[11pt]{amsart}
\usepackage{amsmath,amsthm}
\usepackage{amssymb}
\usepackage[latin1]{inputenc}
\usepackage{enumerate}
\usepackage[left=4cm,top=5cm,bottom=3 cm,right=4cm]{geometry}
\usepackage{pb-diagram}
\theoremstyle{plain}
\newtheorem{theo}{Theorem}
\newtheorem{defin}[theo]{Definition}
\newtheorem{lemma}[theo]{Lemma}

\newtheorem{coro}[theo]{Corollary}

\newcommand{\pdo}{\Psi{\rm DO}}

\title{ K-Theories for Classes of Infinite Rank Bundles}
\author{Andr\'es Larra\'in-Hubach}
\thanks{ The author was supported by a GAANN Fellowship from the US Department of Education.}
\address{Department of Mathematics and Statistics, Boston University, Boston, MA, USA}
\email{alh@bu.edu}

\begin{document}
\bibliographystyle{plain}

\begin{abstract}
Several authors have recently constructed characteristic classes for classes of  infinite rank vector bundles appearing in topology and physics.  These include
the tangent bundle to the space of maps between closed manifolds, the infinite rank bundles in the families index theorem, and bundles with
pseudodifferential operators as structure group.
In this paper, we construct the corresponding K-theories for these types of bundles.
We
develop the formalism of these theories and use their Chern character
to
detect a large class of nontrivial elements.
\end{abstract}

\maketitle

\section{\textbf{Introduction}}
In this paper, we define   K-theories for three classes of infinite rank vector bundles called gauge bundles, pseudodifferential bundles,  and families bundles.   For example, gauge bundles appear naturally when studying the tangent bundle to the
space of maps between closed manifolds, e.g. in string theory.  Constructing the  Levi-Civita connection on these tangent bundles forces  the extension of the structure group to a group of bounded
pseudodifferential operators  ($\pdo$s) \cite{RMT}.  Families bundles arise in the setup of the Atiyah-Singer families index theorem, as mentioned in \cite{AS4} and used by Bismut in his local
proof of the families index theorem \cite[Ch. 10]{BGV}; the families index theorem has been used to detect anomalies in quantum field theory \cite{AS, BF}.

In \cite{ALH1}, \cite{Pay}, so-called leading order Chern classes  were defined and used to
find   nontrivial examples of gauge and pseudodifferential bundles. The existence of a leading order
 Chern character is the  motivation for this paper, as there should exist K-theories corresponding to  these
 bundles as the natural domain of this Chern character.

There is a well known difficulty to constructing K-theory for infinite rank bundles with Hilbert
space fibers:
Hilbert bundles over CW complexes are trivial, because the structure
group ${\rm GL}(\mathcal H)$
of invertible bounded linear operators on a complex Hilbert space $\mathcal H$
 is contractible. However,
topologically or geometrically interesting subgroups
of ${\rm GL}(\mathcal H)$ may have nontrivial topology, and so may lead to interesting
K-theories.
In \cite{RMT}, \cite{Pay}, certain Hilbert bundles with restricted structure groups are defined, and  nontrivial examples appear in \cite{ALH1}. To construct these bundles, copies of a
Sobolev space $\Gamma^s(N,E)$ of sections of a finite rank vector bundle $E$
over a closed manifold $N$ are glued over a CW complex $X$, using elements of the invertible zero order
$\pdo$s  $\Psi_0^*(N,E)$ as transition maps. We call the resulting Hilbert bundles {\it pseudodifferential bundles}.  As a special case, we can glue copies of $\Gamma^s(N,E)$
using smooth gauge transformations of $E$.  We call these bundles {\it gauge bundles}.

To form characteristic classes for $\pdo$ and gauge
bundles by Chern-Weil theory, we need a trace
on the
Lie algebra $\Psi_0(N,E)$ of the structure group consisting of  zero order $\pdo$s.  These
traces are basically of two types: the Wodzicki residue and the leading order trace (given
by integrating the leading order symbol over the cosphere bundle) \cite{LN}. Using the Wodzicki residue, one can define ``Wodzicki-Chern classes";   however,  these always vanish \cite{ALH2}. The leading order trace gives rise to ``leading order Chern classes."
Several examples of bundles with nontrivial leading order classes are known, which
indicates that the corresponding K-theories $K^\Psi$ of $\pdo$-bundles
and $K^{\mathcal G}$ of gauge bundles should be nontrivial.

Neither of these K-theories handles the infinite rank bundles that arise in the families index
theorem.  As explained in \cite{AS4}, the appropriate structure group is
 $\textrm{Diff}(N,E)$, consisting of pairs $(\phi,f)$, where $\phi$ is a diffeomorphism of $N$
 and $f$ is a bundle isomorphism of $E$ covering
   $\phi$. Note that the gauge group of $E$ is the subgroup of $\textrm{Diff}(N,E)$ where
   $\phi = {\rm Id},$ and that $\textrm{Diff}(N,E)$ is a subgroup of
   ${\rm GL}(\Gamma^s(N,E)).$  The corresponding K-theory
   of {\it families bundles} is denoted $K^{\textrm{Diff}}$.

In \S2, we set up the foundations for $K^\mathcal{G}$ and $K^{\textrm{Diff}}$.  Once the
topology of  the structure groups is fixed, the constructions are fairly straightforward.    The
main point is to let our bundles $E\to N$ vary over both $E$ and $N$ in order to form good
sums and products.
 In Lemmas \ref{lemimp1} and \ref{Ktheqgau}, we show that  $K^{\textrm{Diff}}$ and $K^\mathcal{G}$ are isomorphic to specific ordinary K-theory rings. For example, an element of $K^\mathcal{G}(X)$ can be
 represented by an element of $K(X\times N)$ for some $N$. These results are related to the
 generalized caloron construction in \cite{HMV}. In Theorem \ref{omegaspecimp}, we show that $K^\mathcal{G}$ is the first term of a generalized cohomology theory which does not have Bott periodicity. In Theorems  \ref{Thom} and
 \ref{Serre Swan},  we get analogues of the  Thom isomorphism and the Serre-Swan theorems.

 In \S3, we compare the leading order Chern character of an element of $K^\mathcal{G}(X)$ to the ordinary
 Chern character of the corresponding element in $K(X\times N).$  This is used to detect
 nontrivial elements in $K^\mathcal{G}(X).$

 In \S4, we define the ring $K^\Psi$ for $\pdo$ bundles.  Again, we can use the leading order
 Chern character to show that $K^\Psi(X)$ is large.  More precisely,
 for any $N$ as above, in Theorem \ref{lastthm} and Corollary \ref{finalcoro} we show that almost all of $K(X\times N)$ injects into  $K^\Psi(X)$.

\textbf{Acknowledgements.} I would like to thank  Steve Rosenberg for his invaluable help in improving the content and  presentation of  this paper.

\section{\textbf{The K-theory groups $K^{\textrm{Diff}}$ and $K^\mathcal{G}$}}
In this section we construct the rings $K^{\textrm{Diff}}$ and $K^\mathcal{G}$. For a compact CW complex $X$, we completely characterize
$K^{\textrm{Diff}}(X), K^\mathcal{G}(X)$ in terms of ordinary K-theory rings. $K^\mathcal{G}$ will also be used in \S4 to construct nontrivial elements in the  more complicated K-theory $K^\Psi$.
We also establish some fundamental properties of  $K^{\textrm{Diff}}$ and $K^\mathcal{G}$,
including their
extension to a generalized cohomology theory, the Thom isomorphism and a Serre-Swan theorem.
\subsection{\textbf{The Groups $\textrm{Diff}$ and $\mathcal{G}$}}
Throughout this paper, $N$ is a closed orientable Riemannian manifold  and $E\to N$ is a finite rank hermitian vector bundle. Define the group $\textrm{Diff}(N,E)$ to be pairs of maps $(\phi,f)$, with $\phi\in \textrm{Diff}(N)$,  $f\in \textrm{Diff}(E)$ is  linear on the fibers, and such that the following diagram commutes
\begin{align*}
\begin{diagram}
\node{E}\arrow{s,l}{}\arrow{e,t}{f}\node{E}\arrow{s,r}{}\\
\node{N}\arrow{e,t}{\phi}\node{N}
\end{diagram}
\end{align*}
$\textrm{Diff}(N,E)$ is a topological group  with the Fr\'echet topology \cite{AS4}.

Let $\Gamma^s(N,E)$ be the Sobolev completion of the space of smooth sections of $E$ with
respect to a fixed parameter $s$ large enough
  so that elements of  $\Gamma^s(N,E)$ are continuous. An element $(\phi,f)$ acts linearly on $r\in \Gamma^s(N,E)$ by
\begin{equation}\label{one}((\phi,f)\cdot r)(y)=f\cdot r(\phi^{-1}(y)),
\end{equation}
for any $y\in N$.

Take $((\phi_n,f_n))_{n=1}^\infty$ a sequence in $\textrm{Diff}(N,E)$ converging in the Fr\'echet topology to $(\phi,f)$. It is immediate that for $r\in \Gamma^s(N,E)$, $(\phi_n,f_n)r\to (\phi,f)r$ in the $\Gamma^s$ topology. Likewise, if $r_n\to r$, in the $\Gamma^s$ topology, then $(\phi,f)r_n\to (\phi,f)r$, in the $\Gamma^s$ topology. This shows that  $\textrm{Diff}(N,E)$  injects continuously into the space of bounded invertible operators
${\rm GL}(\Gamma^s(N,E))$ on $\Gamma^s(N,E)$ with the norm topology. The group $\mathcal{G}(N,E)=\{(\textrm{id},f)\in \textrm{Diff} (N,E)\}$, of smooth gauge transformations of $E\to N$ is a closed normal subgroup of $\textrm{Diff}(N,E)$.

We can also consider the topology of uniform convergence on $\textrm{Diff} (N,E)$. In this case, $\textrm{Diff} (N,E)$ is dense in the group $\textrm{Homeo}(N,E)$ of pairs $(\phi,f)$ as above, but with $\phi$ and $f$ homeomorphisms. For the uniform topology, $\textrm{Homeo}(N,E)$ injects continuously into  ${\rm GL}(\Gamma^s(N,E))$
 with the norm topology. With the uniform topology, $\mathcal{G}(N,E)$
 is a nonclosed  subgroup of $\textrm{Diff} (N,E)$.
  The uniform topology is needed in \S\ref{symbsubsec} to get a continuous symbol map.

\subsection{\textbf{Infinite Rank Bundles Related to Fibrations and $K^{\textrm{Diff}}$}}

Let $M\stackrel{\pi}{\to} X$ be a locally trivial fibration, where $X$ and $M$ are closed, oriented manifolds,
with fibers diffeomorphic to a fixed closed oriented manifold $N$. A finite rank
vector bundle $\mathbb{E}\to M$ induces an infinite rank Fr\'{e}chet bundle $\pi_*(\mathbb{E})\to X$: given $b\in X$, the fiber of $\pi_*(\mathbb{E})$ over $b$ is  $\Gamma(M_b,\mathbb{E}_b)$, the space of smooth sections of the vector bundle $\mathbb{E}_b
= \mathbb{E}|_{\pi^{-1}(b)}\to M_b= \pi^{-1}(b)$ over $b$. Each bundle $\mathbb{E}_b\to M_b$ is noncanonically isomorphic to a fixed vector bundle $E\to N$, called the \textit{local model}. The transition maps for $\pi_*(\mathbb{E})$ are in  $\textrm{Diff}(N,E)$. There is a canonical identification \cite[p. 277]{BGV}
\begin{align}\label{corrsect}
\Gamma(X,\pi_*(\mathbb{E}))\cong\Gamma(M,\mathbb{E}).
\end{align}

It is easier to work with Hilbert bundles. Fix a hermitian metric on $\mathbb{E}$ and a metric on $M$, and redefine $\pi_*(\mathbb{E})$ to be the bundle whose fibers are the
Sobolev sections $\Gamma^s(M_b,\mathbb{E}_b)$. $\pi_*(\mathbb{E})$ is  a Hilbert bundle with fibers isomorphic to $\Gamma^s(N,E)$. Even though we are working with Hilbert bundles, the transition maps lie  in  the restricted subgroup $\textrm{Diff}(N,E)
\subset {\rm GL}(\Gamma^s(N,E))$, which has a highly nontrivial topology.  Thus
$\pi_*(\mathbb{E})$ may be nontrivial as a $\textrm{Diff}(N,E)$ bundle.

More generally, let  $X$ be a CW-complex and $\{U_\alpha\}$ an open covering of $X$. Let
$\{(\phi_{\alpha\beta},f_{\alpha\beta}):U_\alpha\cap U_\beta\to\textrm{Diff}(N,E)\} $
be a cocycle; that is, for $b\in U_\alpha\cap U_\beta\cap U_\gamma$,  $$(\phi_{\alpha\beta},f_{\alpha\beta})\cdot(\phi_{\beta\gamma},
f_{\beta\gamma})=
(\phi_{\alpha\gamma},f_{\alpha\gamma}).$$
We glue $U_\alpha\times \Gamma^s(N,E)$ to $U_\beta\times \Gamma^s(N,E)$ via the transition maps $(\phi_{\alpha\beta},f_{\alpha\beta})$ and call the resulting Hilbert bundle $\mathcal{E}\to X$.  Simultaneously,  the $\{\phi_{\alpha\beta}\}$ glue the $U_\alpha \times
N$ to  a fibration $M\to X$ with fibers diffeomorphic to $N$, and the $\{f_{\alpha\beta}\}$
similarly determine a finite rank vector bundle $\mathbb{E}\to M$. It is easy to see that $\mathcal{E}=\pi_*(\mathbb{E})$. The bundle $\mathcal{E}$ will be called a
{\it families bundle.}

Given families bundles $\pi_*(\mathbb{E})$ and $\pi_*(\mathbb{F})$ over $X$, we only
consider morphisms that are induced by bundle maps from $\mathbb{E}$ to $\mathbb{F}$. In particular, if  $\Xi:\pi_*(\mathbb{E})\to \pi_*(\mathbb{F})$ is an isomorphism, then $\Xi$ is induced from a bundle isomorphism
$\Xi':\mathbb{E}\to \mathbb{F}$, and   $\Xi^{-1}$ is induced from $(\Xi')^{-1}$; thus
$\mathbb{E}$ and $\mathbb{F}$ are isomorphic. We define  $\textrm{Vect}_M^{\textrm{Diff}}(X)$  to be the category of isomorphism classes of vector bundles of the form $\pi_*(\mathbb{E})$, for $\mathbb{E}\in {\rm Vect}(M)$, the category of  isomorphism classes of vector bundles over $M$. Note that  $M$ is fixed. It follows that there is an equivalence of categories between $\textrm{Vect}_M^{\textrm{Diff}}(X)$ and $\textrm{Vect}(M)$.

The transition maps of bundles in $\textrm{Vect}_M^{\textrm{Diff}}(X)$ lie in the group
\begin{align*}
\textrm{Diff}_N(X)=\prod_{E\in \textrm{Vect}(N)}\textrm{Diff}(N,E),
\end{align*}
but a particular bundle has transition maps with only one fixed component not of the form
$({\rm id}, {\rm id}).$

We  add $\pi_*(\mathbb{E})$ and $\pi_*(\mathbb{F})$ in $\textrm{Vect}_M^{\textrm{Diff}}(X)$ by defining $\pi_*(\mathbb{E})\oplus\pi_*(\mathbb{F})=\pi_*(\mathbb{E}\oplus\mathbb{F})$. If $\{(\phi_{\alpha\beta},e_{\alpha\beta})\}$ and $\{(\phi_{\alpha\beta},f_{\alpha\beta})\}$ define $\pi_*(\mathbb{E})$ and $\pi_*(\mathbb{F})$ respectively, then the transition maps for
$\pi_*(\mathbb{E})\oplus\pi_*(\mathbb{F})$
are $\{(\phi_{\alpha\beta},e_{\alpha\beta}\oplus f_{\alpha\beta})\}$. The tensor product is defined by
$\pi_*(\mathbb{E})\otimes\pi_*(\mathbb{F})=\pi_*(\mathbb{E}\otimes\mathbb{F})$,
with transition maps $\{(\phi_{\alpha\beta},e_{\alpha\beta}\otimes f_{\alpha\beta} )\}$. The fibers of $\pi_*(\mathbb{E})\otimes\pi_*(\mathbb{F})$ are isomorphic to $\Gamma^s(N,E\otimes F)$, where $E\to N$ and $F\to N$ are the local models of $\mathbb{E}$ and $\mathbb{F}$, respectively. Notice that this is not the same as taking the tensor product of the fibers of $\pi_*(\mathbb{E})$ and $\pi_*(\mathbb{F})$. With this sum and tensor product, $\textrm{Vect}_M^{\textrm{Diff}}(X)$ becomes an abelian semiring and $\pi_*:\textrm{Vect}(M)\to \textrm{Vect}_M^{\textrm{Diff}}(X)$ a semiring homomorphism. The usual Grothendieck construction gives an abelian ring denoted by  $K^{\textrm{Diff}}_M(X)$. The following result is valid using the Fr\'echet or the uniform topology on $\textrm{Diff}_N(X)$.

\begin{lemma}\label{lemimp1}
Let $X$ be compact. There is a ring isomorphism
\begin{align*}
\pi_*: K(M)\to K^{\rm{Diff}}_M(X).
\end{align*}
\end{lemma}
Here, $K(M)$ denotes the usual K-theory ring of $M$. In particular, for $X$ compact, $K^{\rm{Diff}}_M(X)$ is an abelian ring with unit.
\begin{proof}
Since $X$ and $N$ are compact,  $M$
is  compact space. Elements of $K(M)$ are formal differences $[\mathbb{E}]-[\mathbb{F}]$ for $\mathbb{E},\,\mathbb{F}\in\textrm{Vect}(M)$.
 The $\pi_*$ functor gives $[\pi_*(\mathbb{E})]-[\pi_*(\mathbb{F})] \in K^{\textrm{Diff}}_M(X)$, and since every element of $K^{\textrm{Diff}}_M(X)$ is of this form, $\pi_*$ is surjective.

If $[\pi_*(\mathbb{E})]-[\pi_*(\mathbb{F})]=0$ then there is a bundle $\pi_*(\mathbb{H})$ such that
$$\pi_*(\mathbb{E}\oplus\mathbb{H})\cong\pi_*(\mathbb{E})\oplus\pi_*(\mathbb{H})
\cong \pi_*(\mathbb{F})\oplus\pi_*(\mathbb{H})\cong\pi_*(\mathbb{F}\oplus\mathbb{H}),
$$
so $\mathbb{E}\oplus\mathbb{H}\cong\mathbb{F}\oplus\mathbb{H}$ and $[\mathbb{E}]-[\mathbb{F}]=0$ in $K(M)$.
\end{proof}

Note that if $N=\{\rm{pt}\}$ is a point, then $M=X$ and $\textrm{Vect}_M^{\textrm{Diff}}(X)=\textrm{Vect}(X)$.  In this case, $K^{\textrm{Diff}}_M(X)$ is just $K(X)$.

Define $K^{\textrm{Diff}}(X)$ by
\begin{align*}
K^{\textrm{Diff}}(X)=\bigotimes_{M\to X}K^{\textrm{Diff}}_M(X),
\end{align*}
where the tensor product of rings is taken over all possible diffeomorphism classes
$\Lambda$ of locally trivial smooth fibrations $M\to X$. An element $\otimes_{i\in \Lambda}a_i\in K^{\textrm{Diff}}(X)$ has almost  all $a_i$ equal to the identity.  Thus
a general element of $K^{\textrm{Diff}}(X)$  is a finite sum of elements of the form $a_1\otimes\ldots\otimes a_k$ with $a_j\in K^{\textrm{Diff}}_{M_j}(X)$ for a fibration $M_j\to X$. Tensor products in $K^{\textrm{Diff}}(X)$ are taken componentwise.

\begin{coro}
Let $X$ be compact. There is a ring isomorphism
\begin{align*}
K^{\rm{Diff}}(X)\simeq \bigotimes_{M\to X} K(M).
\end{align*}
\end{coro}

\subsection{\textbf{$K^\mathcal{G}$-Theory.}}
The construction in \S2.2 can be repeated for the subgroup   $\mathcal{G}(N,E)$ of gauge
transformations. For reasons explained later, we only consider the uniform topology.
The associated fibration is trivial, $M=X\times N$, and the induced bundles, called
{\it $\mathcal{G}$-vector bundles} or {\it gauge bundles}, are of the form $\mathcal{E}=\pi_*(\mathbb{E})$ for $\mathbb{E}\to X\times N$ a finite rank complex vector bundle. The set of isomorphism classes of $\mathcal{G}$-vector bundles is denoted by $\textrm{Vect}^\mathcal{G}_N(X)$. As before, there is an equivalence of categories between $\textrm{Vect}^\mathcal{G}_N(X)$ and $\textrm{Vect}(X\times N)$.

Set
\begin{align*}
\mathcal{G}_N=\prod_{E\in \textrm{Vect}(N)}\mathcal{G}(N,E).
\end{align*}
Every $\mathcal{G}$-vector bundle will have its transition functions in one term of this disjoint union. Since the base $N$ is fixed, we denote $\mathcal{G}_N$ by $\mathcal{G}$.

The previous lemma holds in this context:
\begin{lemma}\label{Ktheqgau}
Let $X$ be compact. There is a ring isomorphism
\begin{align*}
\pi_*: K(X\times N)\to K_N^\mathcal{G}(X).
\end{align*}
\end{lemma}

For example, if $N=S^1$ and $X$ is compact, then
\begin{align*}
K_{S^1}^\mathcal{G}(X)=K(X\times S^1)\cong K^0(X)\oplus K^1(X),
\end{align*}
\cite[p. 110]{EfP}.

The group $\mathcal{G}(N,E)$ with the uniform topology is not a topological group, but it is uniformly dense in the topological group of continuous gauge transformations of $E$ with the uniform topology. We can construct $\mathcal{G}$-vector bundles as before, using continuous gauge transformations as transition maps. (By standard approximations,  the transition maps of any of these bundles can be assumed to be smooth.) In contrast to
smooth gauge transformations, the group of  continuous gauge transformations, also denoted by $\mathcal{G}(N,E)$ for the rest of this section, is a topological group and has
 an
 explicit model for its classifying space \cite{AB1}:
 \begin{align*}
B\mathcal{G}(N,E)=\textrm{Maps}_E(N,BU),
\end{align*}
where $\textrm{Maps}_E(N,BU)=\{f:N\to BU, f^*(EU)=E\}$. This
is stated in \cite{AB1} for principal bundles, but holds for vector bundles associated to
the  faithful representation (\ref{one}). In summary, we switch to  continuous
gauge transformations in this section to use the
Atiyah-Bott construction, but just as with ordinary bundles, it  makes no difference if we consider smooth or continuous gauge transformations as transition maps.

For $f\in \textrm{Maps}_E(N,BU)$, every map homotopic to $f$ is also in $\textrm{Maps}_E(N,BU)$, so $\textrm{Maps}_E(N,BU)$ is a path connected component of $\textrm{Maps}(N,BU)$.
There is a bijection from $\textrm{Vect}^\mathcal{G}_N(X)$ to $[X,B\mathcal{G}]$,
the set of homotopy classes of maps from $X$ to  $B\mathcal{G} = B\mathcal{G}(N,E)$.
Thus
$$
B\mathcal{G}=\coprod_{E\in \textrm{Vect}(N)}B\mathcal{G}(N,E)
=\coprod_{E\in \textrm{Vect}(N)}\textrm{Maps}_E(N,BU).$$
Since every map $f:N\to BU$ lies in some component $\textrm{Maps}_F(N,BU)$,
\begin{align}
B\mathcal{G}= \textrm{Maps}(N,BU).
\end{align}
\begin{theo}
There is a bijective correspondence
\begin{align*}
K^\mathcal{G}_N(X)\simeq[X,B\mathcal{G}\times \mathbb{Z}],
\end{align*}
for $\mathcal{G} = \mathcal{G}(N,E).$
\end{theo}
Here $\simeq$ is used to denote a bijection, but  it induces a tautological ring isomorphism.
\begin{proof}
Any homotopy $\phi:X\times[0,1]\to \textrm{Maps}(N, BU\times\mathbb{Z})$ can be seen as a map $\phi:X\times N\times[0,1]\to BU\times\mathbb{Z}$. Therefore $[X,\textrm{Maps}(N,BU)\times\mathbb{Z}]=[X\times N,BU\times\mathbb{Z}]$. Thus
\begin{align*}
K^\mathcal{G}_N(X)&\cong K(X\times N)\simeq[X\times N,BU\times \mathbb{Z}]\\
&\simeq[X,\textrm{Maps}(N,BU)\times \mathbb{Z}]\simeq[X,B\mathcal{G}\times \mathbb{Z}].
\end{align*}

\end{proof}

We can also build  $\mathcal{G}(N,E)$-principal bundles $R\to X$ using cocycles
$
\kappa_{\alpha\beta}:U_\alpha\cap U_\beta\longrightarrow \mathcal{G}(N,E),$
for $\{U_\alpha\}$ a covering of $X$.
We denote the set of $\mathcal{G}$-principal bundles by $\textrm{Prin}_N^{\mathcal{G}}(X)$. The canonical representation of $\mathcal{G}(N,E)$ on the space $\Gamma^s(N,E)$
gives an associated $\mathcal{G}$-vector bundle
\begin{align*}
\mathcal{E}=R\times_{\mathcal{G}(N,E)}\Gamma^s(N,E).
\end{align*}
$\mathcal{G}(N,E)$-principal bundles are used in section \S\ref{symbsubsec} to define the symbol map.

\subsection{\textbf{ $K^{\mathcal{G}}$-Theory and $\Omega$ Spectra}}

We  recall the definition of an $\Omega$ spectrum.
\begin{defin}
A sequence of based CW complexes and base point preserving maps $(E_k,\epsilon_k)_{k\in\mathbb{Z}}$ is an $\Omega$ spectrum if each  $\epsilon_k:E_k\to \Omega E_{K+1}$ is a homotopy equivalence.
\end{defin}

Given a CW complex $X$ and an $\Omega$ spectrum $(E_k,\epsilon_k)_{k\in\mathbb{Z}}$, the spaces $[X,E_k]$  form a generalized cohomology theory satisfying Milnor's additivity axiom. Any additive generalized cohomology theory has an $\Omega$ spectrum \cite[p. 35]{Ko}.

We wish to show that $B\mathcal{G}\times \mathbb{Z}$  is the first term in an $\Omega$ spectrum. First we prove
\begin{lemma}\label{Omegaspectmaps}
Given an $\Omega$ spectrum $(E_k,\epsilon_k)_{k\in\mathbb{Z}}$ and a based compact CW complex $X$, there exist based maps $\epsilon_k^*:\textrm{Maps}(X,E_k)\to\textrm{Maps}(X,E_{k+1})$ such that the sequence $(\textrm{Maps}(X,E_k),\epsilon^*_k)_{k\in\mathbb{Z}}$ is an $\Omega$ spectrum.
\end{lemma}
\begin{proof}
The spaces $\textrm{Maps}(X,E_k)$ are homotopy equivalent to CW complexes. By hypothesis, there is a homotopy equivalence $E_2\sim\Omega E_1$, so
\begin{align*}
\textrm{Maps}(X,E_2)\sim\textrm{Maps}(X,\Omega E_1)=\textrm{Maps}_*(X\times S^1,E_1),
\end{align*}
where $\textrm{Maps}_*(X\times S^1,E_1)$ are  maps sending $X\times \{1\}$ to a fixed base point. This implies
\begin{align*}
\textrm{Maps}_*(X\times S^1,E_1)=\Omega(\textrm{Maps}(X,E_1)).
\end{align*}
We proceed inductively: given $\epsilon_k:E_k\to E_{k+1}$ and a map $f:X\to E_k$, we define $\epsilon^*_k(f)=f\circ \epsilon_k$.
\end{proof}
Recall that complex K-theory is the extraordinary cohomology theory associated to the two-periodic spectrum  $\Omega^K=(BU\times\mathbb{Z},U,BU\times\mathbb{Z},\ldots)$.
\begin{theo}\label{omegaspecimp}
Given a compact CW complex $X$, the group $K_N^\mathcal{G}(X)\cong[X,B\mathcal{G}\times \mathbb{Z}]$ is the first term of a two periodic extraordinary cohomology theory.
\end{theo}
\begin{proof}
By Theorem 4,
\begin{align*}
K_N^\mathcal{G}(X)\simeq[X,B\mathcal{G}\times \mathbb{Z}]
\cong[X,\textrm{Maps}(N,BU\times\mathbb{Z})].
\end{align*}
By Lemma \ref{Omegaspectmaps}, $\Omega_N^\mathcal{G}=(\textrm{Maps}(N,BU\times\mathbb{Z}),\textrm{Maps}(N,U),\textrm{Maps}(N,BU\times\mathbb{Z})\ldots )$ is an $\Omega$ spectrum. Since $\Omega^K$ is two periodic, so is $\Omega^\mathcal{G}_N$.
\end{proof}

For a single gauge group $\mathcal{G}(N,E)$, we can define another generalized cohomology theory, analogous to connective K-theory.
\begin{lemma}
There is an $\Omega$ spectrum $(B\mathcal{G}(N,E),E_1,E_2,\ldots,\epsilon_k)_{K\in\mathbb{N}}$. Therefore, $[X,B\mathcal{G}(N,E)]$ can be extended to a generalized cohomology theory.
\end{lemma}
\begin{proof}
$(BU,SU,BSU,\ldots)$ are the first three terms of an $\Omega$ spectrum $ku_*$, whose associated extraordinary cohomology theory is connective K-theory \cite{Mor}. For the corresponding spaces of maps we get
\begin{align*}
\textrm{Maps}(N,BU)\sim \Omega \textrm{Maps}(N, SU)\sim \Omega^2 \textrm{Maps}(N,BSU)\ldots
\end{align*}
These homotopy equivalences are  valid for the component $\textrm{Maps}_E(N,BU)$,
so we get an $\Omega$ spectrum with first term $B\mathcal{G}(N,E)$.
\end{proof}

\subsection{\textbf{Bott Periodicity and  Thom Isomorphism in $K^{\mathcal{G}}$}}

Define $K^{\mathcal{G}}(X)$ by
\begin{align*}
K^{\mathcal{G}}(X)=\bigotimes_{N}K^{\mathcal{G}}_N(X),
\end{align*}
where the tensor product of rings is taken over all possible diffeomorphism classes $\Xi$ of orientable manifolds $N$. An element $\otimes_{i\in \Xi}a_i\in K^\mathcal{G}(X)$ has almost all $a_i$ equal to the identity. Thus a general element of $K^\mathcal{G}$ is a finite sum of terms of the form $a_1\otimes\ldots\otimes a_k$ where $a_j\in K_{N_j}^{\mathcal{G}}(X)$ for a fibration $X\times N_j\to X$. Tensor products in $K^{\mathcal{G}}(X)$ are taken componentwise.

For a compact CW complex $X$, by Lemma \ref{Ktheqgau} and $K(X\times S^2)=K(X)\otimes K(S^2)$ \cite[p. 128]{EfP},  $$K^\mathcal{G}_N(X\times S^2)\cong K(X\times S^2\times N)=K(X\times N)\otimes K(S^2)
=K^\mathcal{G}_N(X)\otimes K(S^2).$$
This multiplicativity result is the analogue of Bott periodicity holding in this context.

Let $F\to X$ be a finite rank hermitian vector bundle. The projection $\pi_1:X\times N\to X$  induces $\pi_1^*F\to X\times N$. Let $B(F)$ and $S(F)$ be the associated ball and sphere bundle, respectively. The one point compactification of  $F$, denoted $F^+=F\cup\{\infty\}$, is homeomorphic to the quotient space $B(F)/S(F)$. It follows that $\pi^*(F)^+=F^+\times N$. Using the K-theory Thom isomorphism for  $\pi_1^*F$, we get
\begin{align*}
K^\mathcal{G}_N(X)\cong K(X\times N)\cong\tilde{K}(\pi_1^*(F)^+)\cong \tilde{K}(F^+\times N).
\end{align*}
As in ordinary K-theory, define
\begin{align*}
K^\mathcal{G}_N(F)&=\rm{ker}(K^\mathcal{G}_N(F^+)\to K(\infty))\\
K^\mathcal{G}(F)&=\bigotimes_N K^\mathcal{G}_N(F).
\end{align*}

It is immediate that
$K^\mathcal{G}_N(F)\cong \tilde{K}(F^+\times N)$, so we get a restatement of the Thom isomorphism in the $K^\mathcal{G}$ context:
\begin{theo}\label{Thom}
Let $F\to X$ be a finite rank vector bundle and $X$ a compact CW complex. Then
$K^\mathcal{G}_N(X)\cong K^\mathcal{G}_N(F),$ and so
$$K^\mathcal{G}(X)\cong K^\mathcal{G}(F).$$
\end{theo}

\subsection{\textbf{The Serre-Swan Theorem in $K^\mathcal{G}$}}

For a closed manifold $Y$, let $C(Y)$ denote the ring of complex valued continuous functions. By the classical Serre-Swan theorem, the global sections functor gives an equivalence of categories between $\textrm{Vect}(Y)$ and the category of finitely generated projective modules over $C(Y)$. For $X$ a compact manifold, we want
an equivalence of $\textrm{Vect}^\mathcal{G}_N(X)$ with some category of modules over $C(X)$.

An element $\mathbb{E}\in\textrm{Vect}(X\times N)$ can be seen as a continuous family $\{\mathbb{E}_b\to N\}_{b\in X}$ of vector bundles over $N$, parameterized by $X$. We use the notation $\{\mathbb{E}_b\}\in V_X(N)$. $\Gamma(\{\mathbb{E}_b\})$ denotes the space of maps $h:X\to \{\Gamma(N,\mathbb{E}_b)\}_{b\in X}$ such that $h(b)\in \Gamma(N,\mathbb{E}_b)$ and the family $\{h(b)\}_{b\in X}$ is the restriction of  a continuous section of $\mathbb{E}\to X\times N$. Denote by $\Gamma(V_X(N))$ the category of all spaces of the form $\Gamma(\{\mathbb{E}_b\})$ with morphisms induced by maps between the corresponding vector bundles over $X\times N$.

The following theorem formalizes the idea that ${\rm{Vect}}^\mathcal{G}_N(X)$ should be equivalent to families of finitely generated projective $C(N)$-modules parameterized by $X$.

\begin{theo}\label{Serre Swan}
The global sections functor is an equivalence of categories between ${\rm{Vect}}^\mathcal{G}_N(X)$ and $\Gamma(V_X(N))$.
\end{theo}
\begin{proof}
Take $\mathcal{E}\to X$ in ${\rm{Vect}}^\mathcal{G}_N(X)$. There is a bundle $\mathbb{E}\to X\times N$ with $\pi_*(\mathbb{E})=\mathcal{E}$. We need to show that $\pi_*(\mathbb{E})\longrightarrow \Gamma(\pi_*(\mathbb{E}))$ gives the desired equivalence.

Given $\phi\in \Gamma(X,\pi_*(\mathbb{E}))$ and $b\in X$, $\phi(b)\in \Gamma({\{b\}\times N,\mathbb{E}_b})$. Identifying $\{b\}\times N$ with $N$, we get
$\phi\in\Gamma(\{\mathbb{E}_b\}).$

Conversely, if $h\in \Gamma(\{\mathbb{E}_b\})$, the corresponding continuous family $\{\mathbb{E}_b\}_{b\in X}$ comes from a vector bundle $\mathbb{E}\to X\times N$ with
 $h\in \Gamma(X,\pi_*(\mathbb{E}))$.
\end{proof}

\section{\textbf{The Chern Character in $K^\mathcal{G}$}}
In this section, we show that the leading order Chern character gives a ring homomorphism
from $K^\mathcal{G}(X)$ to de Rham cohomology $H^*(X)$. We use a version of Chern-Weil theory, so from now on we assume that the base $X$ is a closed, orientable, finite dimensional manifold. In this section, $\mathcal{G} = \mathcal{G}(N,E)$ denotes smooth gauge transformations.

\subsection{\textbf{Definition and Basic Properties}}

There is a natural way to define Chern classes and a Chern character on $\mathcal{G}$-vector bundles $\mathcal{E}\to X$, with fibers $\Gamma^s(N,E)$. $X$ admits a partition of unity, so we can put a $\mathcal{G}$-connection on
$\mathcal{E}$ with curvature form $\Omega^\mathcal{E}\in \Lambda^2(X, \Gamma^s(\textrm
{End}(E)))$ taking values in sections of smooth endomorphisms. Note that $\Gamma^s(\textrm{End}(E))$ is formally the Lie algebra of the structure group $\mathcal{G}(N,E)$.

For a choice of Riemannian metric on $N$, there is a natural trace on this Lie algebra, obtained by taking the usual fiberwise trace and then integrating over $N$: for $H\in \Gamma(\textrm{End}(E))$,
\begin{align*}
\textrm{Tr}^\mathcal{G}(H)=\int_N\textrm{tr}\, H_y\,\,\textrm{dvol}(N).
\end{align*}

As in the standard  Chern-Weil construction, the leading order Chern classes $c_k^\mathcal{G}(\mathcal{E})$ and the leading order Chern character $ch^\mathcal{G}(\mathcal{E})$ are the de Rham cohomology classes \cite{Pay}
\begin{align*}
c_k^\mathcal{G}(\mathcal{E})&=[\textrm{Tr}^\mathcal{G}(\Omega^\mathcal{E})^k]
\in H^{2k}(X),\\
ch^\mathcal{G}(\mathcal{E})&=[\textrm{Tr}^\mathcal{G}(\exp(\Omega^\mathcal{E}))]
\in H^{\rm{ev}}(X).
\end{align*}

The term ``leading order'' will be explained in the next section.
\begin{lemma}
Let  $\mathcal{E}$ and $\mathcal{F}$ be $\mathcal{G}$-vector bundles with transition maps in $\mathcal{G}(N,E)$ and $\mathcal{G}(N,F)$, respectively. Then
\begin{align*}
ch^\mathcal{G}(\mathcal{E}\oplus \mathcal{F})&=ch^\mathcal{G}(\mathcal{E})+
ch^\mathcal{G}(\mathcal{F}),\\
ch^\mathcal{G}(\mathcal{E}\otimes \mathcal{F})&=ch^\mathcal{G}(\mathcal{E})\,\cup\,ch^\mathcal{G}(\mathcal{F}).
\end{align*}
Thus $ch^\mathcal{G}$ induces a ring homomorphism
$$ch^\mathcal{G}:K^\mathcal{G}_N(X)\to H^{\textrm{\rm{ev}}}(X,\mathbb{C}).$$
\end{lemma}

The additivity of the Chern character implies the Whitney sum formula for the leading order Chern
classes.

\begin{proof}
Given connections $\nabla^\mathcal{E}$ and $\nabla^\mathcal{F}$ on $\mathcal{E}$ and $\mathcal{F}$, we can follow verbatim the usual arguments from finite dimensional Chern-Weil theory and obtain
\begin{align*}
\Omega^{\mathcal{E}\oplus\mathcal{F}}=
\begin{pmatrix}
\Omega^{\mathcal{E}}&0\\
0&\Omega^{\mathcal{F}}
\end{pmatrix}.
\end{align*}
Thus
\begin{align*}
\textrm{Tr}^\mathcal{G}(\exp(\Omega^{\mathcal{E}\oplus\mathcal{F}}))=
\,\,\textrm{Tr}^\mathcal{G}(\exp(\Omega^{\mathcal{E}}))+
\,\,\textrm{Tr}^\mathcal{G}(\exp(\Omega^{\mathcal{F}})),
\end{align*}
from which the additivity of $ch^\mathcal{G}$ follows. A similar argument works for tensor products.
\end{proof}

The usual Chern character in $K(X\times N)$ gives a complex ring isomorphism
\begin{align*}
CH:K^\mathcal{G}_N(X)\otimes\mathbb{C}\to H^{\textrm{ev}}(X\times N,\mathbb{C}).
\end{align*}
\begin{lemma}
 $CH$ and $ch^\mathcal{G}$ are related by
\begin{align*}
ch^\mathcal{G}(\mathcal{E})=\int_NCH(\mathbb{E})\,{\rm{dvol}}(N),
\end{align*}
where $\pi_*(\mathbb{E})=\mathcal{E}$.
\end{lemma}
\begin{proof}
Let $\xi$, $\eta$ be  vector fields on $X$. Their horizontal lifting to $X\times N$ with respect to the canonical splitting $T(X\times N)=TX\oplus TN$ are $\xi^H=(\xi,0)$ and $\eta^H=(\eta,0)$.

A connection $\nabla^{\mathbb{E}}$ on $\mathbb{E}$  induces a connection $\nabla^\mathcal{E}$ on $\pi_*(\mathbb E)=\mathcal{E}$ by \cite[p. 282]{BGV}
\begin{align*}
\nabla^\mathcal{E}_\xi(s)=\nabla^{\mathbb{E}}_{(\xi,0)}\tilde{s},
\end{align*}
where $s\leftrightarrow \tilde{s}$ is the correspondence  $\Gamma(X,\pi_*(\mathbb{E}))=\Gamma(X\times N,\mathbb{E})$. The curvature operators satisfy
\begin{align*}
\Omega^{\mathcal{E}}(\xi,\eta)=
\Omega^{\mathbb{E}}((\xi,0),(\eta,0)).
\end{align*}
Let $x$ denote local coordinates on $X$ and $y$ local coordinates on $N$. Locally,
\begin{align*}
\textrm{tr}(\Omega^{\mathbb{E}})=\sum_{i,j}a_{ij}(x,y)dx^i\wedge dx^j+
\sum_{k,l}b_{kl}(x,y)dx^k\wedge dy^l+
\sum_{p,q}c_{pq}(x,y)dy^p\wedge dy^q,
\end{align*}
for $a_{ij},b_{kl},c_{pq}\in C^\infty(X\times N)$. The exponential will be of the form
\begin{align*}
\rm{tr}(\exp(\Omega^{\mathbb{E}}))&=\sum_{I,J}a'_{IJ}(x,y)dx^I\wedge dx^J\\
&\qquad +\sum_{K,L}b'_{KL}(x,y)dx^K\wedge dy^L+
\sum_{P,Q}c'_{PQ}(x,y)dy^P\wedge dy^Q,
\end{align*}
for multiindices $I,J,K,L,P,Q$. Since $dy^j(\xi,0)=0$ for all $j$,
\begin{align*}
ch^\mathcal{G}(\mathcal{E})&=\int_N\textrm{tr}(\exp(\Omega^{\mathcal{E}}))\textrm{dvol}(N)\\
&=\bigg(\sum_{I,J}\int_Na'_{IJ}(x,y)\textrm{dvol}(N)\bigg)dx^I\wedge dx^J
\end{align*}
and
\begin{align*}
\int_NCH(\mathbb{E})\textrm{dvol}(N)&=\int_N \textrm{tr}(\exp(\Omega^{\mathbb{E}}))\textrm{dvol}(N)\\
&=\bigg(\int_N\sum_{I,J}a'_{I,J}(x,y)\,\textrm{dvol}(N)\bigg)dx^I\wedge dx^J
\end{align*}
\end{proof}

We now treat naturality of the leading order Chern class.
Let $\mathcal{E}=\pi_*(\mathbb{E})\in \textrm{Vect}_N^\mathcal{G}(X)$ and  $f:Y\to X$ a map with $Y$ another closed orientable manifold. The pullback bundle $f^*\mathcal{E}\to Y$ is defined as  $$f^*\mathcal{E}:=\pi_*\big((f\times \textrm{Id})^*(\mathbb{E})\big),$$ where $f\times \textrm{Id}:Y\times N\to X\times N $ is the induced map.
\begin{lemma}

With the same notation as before,
$$ch^\mathcal{G}(f^*\mathcal{E})=f^*ch^\mathcal{G}(\mathcal{E}).$$
\end{lemma}
\begin{proof}
By the previous lemma,
$$ch^\mathcal{G}(f^*\mathcal{E})=\int_N CH\big((f\times \textrm{Id})^*(\mathbb{E})
\big)\textrm{dvol}(N)=\int_N (f\times \textrm{Id})^*CH(\mathbb{E})\textrm{dvol}(N).
$$
Since $(f\times \textrm{Id})^*$ commutes with the integral over the fiber $N$,
$$ch^\mathcal{G}(f^*\mathcal{E})=(f\times \textrm{Id})^*\int_N CH(\mathbb{E})\textrm{dvol}(N)=f^*(ch^\mathcal{G}(\mathcal{E})).$$
\end{proof}

The leading order Chern character extends to $K^\mathcal{G}(X)$ by
setting  
$$ch^\mathcal{G}(a_1\otimes\ldots\otimes a_k)=ch^\mathcal{G}(a_1)
\cup\ldots\cup ch^\mathcal{G}(a_k),$$
and then extending  linearly. With this extension,
$$ch^\mathcal{G}:K^\mathcal{G}(X)\to H^{\textrm{ev}}(X)$$
is again a ring homomorphism.

We now show that this Chern character is compatible with the module structures on K-theory and cohomology.
$K^\mathcal{G}_N(X)$ is a $K^\mathcal{G}_N(\textrm{pt})$-module as follows.
Since $K^\mathcal{G}_N(\textrm{pt})\cong K(N)$ and $K^\mathcal{G}_N(X)\cong K(X\times N)$, we get a product
\begin{align*}
\alpha: K^\mathcal{G}_N(\textrm{pt})\otimes K^\mathcal{G}_N(X)\to
K^\mathcal{G}_N(X),
\end{align*}
by taking $(H\to N)\in K^\mathcal{G}_N(\rm{pt})$ and $\mathbb{E}\to X\times N\in K^\mathcal{G}_N(X)$ and defining
\begin{align*}
\alpha(H,\mathbb{E}):=\pi_2^*(H)\otimes \mathbb{E}\in K(X\times N)\cong K^\mathcal{G}_N(X).
\end{align*}
It is easy to see that the  diagram
\begin{align}\label{impdiag}
\begin{diagram}
\node{K^\mathcal{G}_N(\textrm{pt})\otimes K^\mathcal{G}_N(X)}\arrow{s,l}{CH\otimes\, CH}\arrow{e,t}{\alpha}\node{K^\mathcal{G}_N(X)}\arrow
{s,l}{CH}\\
\node{H^{\textrm{ev}}(N)\otimes H^{\textrm{ev}}(X\times N)}\arrow{e,t}{\cup}\node{H^{\textrm{ev}}(X\times N)}
\end{diagram}
\end{align}
commutes, {\it i.e.}  the $K^\mathcal{G}_N(\rm{pt})$-module structure of $K^\mathcal{G}_N(X)$ and the $H^{\rm{ev}}(N)$-module structure of $H^{\rm{ev}}(X\times N)$ are compatible. Integrating over $N$ gives another commutative diagram, this time for the leading Chern character:
\begin{align}
\begin{diagram}\label{impdiag2}
\node{K^\mathcal{G}_N(\textrm{pt})\otimes K^\mathcal{G}_N(X)}\arrow{s,l}{ch^\mathcal{G}\otimes\, ch^\mathcal{G}}\arrow{e,t}{\alpha}\node{K^\mathcal{G}_N(X)}\arrow
{s,l}{ch^\mathcal{G}}\\
\node{\mathbb{C}\otimes H^{*}(X)}\arrow{e,b}{}\node{H^{*}(X)}
\end{diagram}
\end{align}

There are diagrams corresponding to  (\ref{impdiag}) and (\ref{impdiag2}) for $K^\mathcal{G} = \bigotimes K^\mathcal{G}_N$, given by replacing
the cohomology groups with
$$\bigotimes_{N}H^{\textrm{ev}}(N)\otimes H^{\textrm{ev}}(X\times N)\ {\rm and}
\  \bigotimes_{N}H^{\textrm{ev}}(X\times N).$$

\subsection{\textbf{Examples of Nontrivial $K^\mathcal{G}_{S^1}$-classes}}\label{exloop}

The tangent bundle to the free  loop space of a manifold is a natural example of a $\mathcal{G}$-vector bundle.
Let $X^n$ be a smooth, oriented, closed manifold and $LX$ its loop space.
Here $LX$ is the completion of the space of smooth loops with respect to the Sobolev topology  for large parameter $s$, as
explained in \cite{RMT}. The complexified
tangent bundle $T_\mathbb{C}LX\to LX$ is canonically a $\mathcal{G}(S^1,\underline{\mathbb{C}}^n)$-vector bundle modeled on a trivial bundle $\underline{\mathbb{C}}^n
 = S^1\times \mathbb{C}^n\to S^1$.

There is a canonical inclusion $\iota: X\to LX$ by considering a point as a constant loop. The pullback bundle $\mathcal{F}=\iota^*(T_\mathbb{C}LX)\to X$ is a $\mathcal{G}$-vector bundle. The fiber over a point $x_0\in X$ is given by $\mathcal{F}|_{x_0}=L(T_{x_0}X\otimes\mathbb{C})$.
The transition maps for $\mathcal{F}$ are the the transition maps of $T_\mathbb{C}X$, but now acting on $L\mathbb{C}^n$ instead of $\mathbb{C}^n$. Thus at  $x_0\in X$ in the overlap of two charts, the transition map  in $GL(n,\mathbb{C})$ for $T_{x_0}X\otimes \mathbb C$ is also  a constant function in
\begin{align*}
\textrm{Maps}(S^1, GL(n,\mathbb{C}))=\mathcal{G}(S^1,\underline{\mathbb{C}}^n)
\end{align*}
for $\mathcal F.$
Since the transition maps are constant, we can construct a connection $\nabla^{\mathcal{F}}$ on $\mathcal{F}$ taking values in $\textrm{End}(\mathbb{C}^n)$ (identified with constant maps in
$\textrm{Maps}(S^1, \rm{End}(\mathbb{C}^n))$.
By Lemma 12, The corresponding leading order Chern character is
\begin{align*}
ch^\mathcal{G}(\mathcal{F})=\textrm{vol} (S^1)\,CH(T_\mathbb{C}X).
\end{align*}
Therefore, manifolds with $CH(T_\mathbb{C}X)\neq 0$ give examples of nontrivial
elements $\mathcal{F}\to X$ in $K^\mathcal{G}_{S^1}(X).$

\section{\textbf{K-Theory for Pseudodifferential Bundles }}
As mentioned in the introduction, pseudodifferential bundles arise in the study of the geometry of mapping spaces \cite{RMT}. In fact, the leading order Chern character was originally defined for pseudodifferential bundles in \cite{Pay}.  In this section we construct a K-theory $K^\Psi$ for these bundles with the leading order Chern character again a ring homomorphism
from $K^\Psi$ to de Rham cohomology. The construction of $K^\Psi$ is also
motivated by the discovery of
nontrivial examples of pseudodifferential bundles in \cite{ALH1}.

\subsection{Preliminaries on Pseudodifferential Operators}
We recall some results of H\"{o}rmander \cite{Hor} about the norm closure
of pseudodifferential operators of  order zero. Let $E$ be a finite rank
Hermitian vector bundle over a closed Riemannian manifold $N$. Denote by
$\Psi_0=\Psi_0(N,E)$ the algebra of zero order
pseudodifferential operators acting on
$\Gamma^s(N,E)$. The group of invertible elements of $\Psi_0$
is denoted by $\Psi_0^*$. The leading symbols
of operators in $\Psi_0$ are smooth sections of $\textrm{End}(S^*N,\pi^*E)$,
where $\pi:S^*N\to N$ is the cosphere bundle.

Operators in $\Psi_0$ are bounded on $\Gamma^s(N,E)$.
For $\overline{\Psi}_0$ the norm closure of $\Psi_0$ in $GL(\Gamma^s(N,E))$,
 the leading symbol extends
to a continuous map\\
 $\overline{\Psi}_0\to\textrm{End}(S^*N,\pi^*E)$,  where $\textrm{End}( S^*N,\pi^*E)$ now denotes continuous endomorphisms with the uniform topology.
For $P\in \Psi_0(N,E)$ and a bundle $G\to Y$, the operator
$P\otimes 1$ acting on smooth sections of $E\boxtimes G \to N\times Y$
can be extended to a bounded operator in $\overline{\Psi}_0(E\boxtimes G \to N\times Y)$.
The leading symbol of the extension is $\sigma_0(P)\otimes 1$, where $\sigma_0(P)$
is the principal symbol of $P$ \cite[p. 202]{Hor}.
Note that the extension lies only in the closure
$\overline{\Psi}_0(E\boxtimes G \to N\times Y)$.

Assume that  $G\to Y$ is $ \mathbb{C}\to{\rm{pt}}$,
considered as a complex vector bundle over a point. We get an operator $P\otimes 1$
acting on $E\otimes 1\to N\times \{\rm{pt}\}$, which can be canonically identified
with $E\to N$. Given a pseudodifferential operator $Q\in\Psi_0(N,F)$,
we get an operator $P\otimes Q\in \overline{\Psi}_0(N,E\otimes F)$ defined by
\begin{align}\label{tensorpsido}
P\otimes Q=(P\otimes 1)\circ(1\otimes Q).
\end{align}
Here we are using
$E\otimes F \cong (E\otimes 1)\otimes(1\otimes F)$.

This allows us to define tensor products of pseudodifferential operators.
Again, this product lies only in the norm closure and is not \textit{per
se} a pseudodifferential operator.

\subsection{\textbf{Pseudodifferential Bundles}}
As before, we can construct a vector bundle over $X$ using pseudodifferential operators. Take  an open covering $\{U_\alpha\}$ of $X$ and glue the spaces  $U_\alpha\times \Gamma^s(N,E)$  using transition maps
\begin{align*}
\phi_{\alpha\beta}^\mathfrak{E}:U_\alpha\cap U_\beta\longrightarrow
\overline{\Psi}_0^*(N,E).
\end{align*}
(Here $\overline{\Psi}_0^* = (\overline {\Psi_0})^*$, not $\overline{\Psi_0^*}.$)
Denote the resulting bundle by $\mathfrak{E}\to X$ \cite{Pay}.
For $\mathfrak{E}\to X$ as above,
$E\to N$ is called its \textit{model bundle}, and  $\mathfrak{E}$
is called a $\Psi$-bundle over $X$. A homomorphism between $\Psi$-bundles is a continuous family of fiber preserving linear maps given fiberwise
by elements of the norm closure of  pseudodifferential operators
of order zero. Here, continuity is taken with respect to the norm topology on the space of pseudodifferential operators.

For fixed $N$, we
define $\textrm{Vect}_N^{\Psi}(X)$ to be the set of  isomorphism classes
of $\Psi$-bundles over $X$ for all model bundles $E\to N$.
The full structure group of these bundles is
\begin{align*}
\overline{\Psi}_N^*=\prod_{E\in \textrm{Vect}(N)}\overline{\Psi}_0^*(N,E).
\end{align*}
The closed manifold  $X$   admits partitions of unity, so  elements of $\textrm{Vect}_N^{\Psi}(X)$ admit fiber metrics and connections, and $\Psi$-bundles over contractible relatively
 compact open sets are trivial.

There is a straightforward procedure to make  $\textrm{Vect}_N^{\Psi}(X)$ an abelian
semiring.

For sums, take $\mathfrak{E}\to X$ as above and another $\Psi$-bundle
$\mathfrak{F}\to X$ with fibers $\Gamma^s(N,F)$ and transition maps
 $\phi_{\alpha\beta}^\mathfrak{F}:U_\alpha\cap U_\beta\longrightarrow
 \overline{\Psi}_0^*(N,F)$, with both bundles  trivial over the open sets
 $\{U_\alpha\}$ of a covering of $X$.
The transition maps for  $\mathfrak{E}\oplus \mathfrak{F}\longrightarrow X$
are
\begin{align*}
\phi^{\mathfrak{E}\oplus\mathfrak{F}}_{\alpha\beta}=
\begin{pmatrix}
\phi_{\alpha\beta}^\mathfrak{E}&0\\
0& \phi_{\alpha\beta}^\mathfrak{F}
\end{pmatrix}
\end{align*}
acting on $\Gamma^s(N,E\oplus F)$.

For  $\mathfrak{E}\otimes \mathfrak{F}\longrightarrow X$,
 take
\begin{align*}
\phi_{\alpha\beta}^\mathfrak{E}\otimes \phi_{\alpha\beta}^\mathfrak{F}=
(\phi_{\alpha\beta}^\mathfrak{E}\otimes 1)\circ(1\otimes
\phi_{\alpha\beta}^\mathfrak{F}),
\end{align*}
as in (\ref{tensorpsido}), acting on  $\Gamma^s(N,E\otimes F)$.
It is immediate that the cocycle conditions of the transition
maps are preserved, since the operators $\phi_{\alpha\beta}^
\mathfrak{E}\otimes 1$ and $1\otimes \phi_{\alpha\beta}^\mathfrak{F}$
commute.

From the abelian semiring $\textrm{Vect}_N^{\Psi}(X)$ we can
define the corresponding $K^\Psi$-theory ring $K^\Psi_N(X)$ as usual.
In particular, for $N=\{\rm{pt}\}$,
$\rm{Vect}_{\{\rm{pt}\}}^\Psi(X)=\textrm{Vect}(X)$ so $K^\Psi_{\textrm{pt}}(X)=K(X)$.

Similarly, we can define principal bundles
with fibers isomorphic to $\overline{\Psi}_0^*(N,E)$.
Once again, we glue copies of the group using the
transition maps. These principal bundles
are called $\Psi$-principal bundles and denoted
by $\textrm{Prin}_N^{\Psi}(X)$.

Let $P\to X$ be a $\Psi$-principal bundle. Using the
 canonical faithful representation of $\overline{\Psi}_0^*(N,E)$
 on $\Gamma^s(N,E)$ we get an associated $\Psi$-vector bundle
\begin{align*}
\mathfrak{E}=P\times_{\overline{\Psi}_0^*(N,E)} \Gamma^s(N,E).
\end{align*}
 In this paper, we only consider the action of  $\overline{\Psi}_0^*(N,E)$ on $\Gamma^s(N,E)$, so there is a one-to-one correspondence between $\Psi$-vector bundles and $\Psi$-principal bundles.

As with $K^\mathcal{G}(X)$, we define $K^{\Psi}(X)$ by
\begin{align*}
K^{\Psi}(X)=\bigotimes_{N}K^{\Psi}_N(X),
\end{align*}
where the tensor product of rings is taken over all diffeomorphism classes $\Xi$ of closed orientable manifolds $N$.

\subsection{\textbf{The Symbol Map in $K^\Psi$}}\label{symbsubsec}

The leading
symbol $\sigma_0(P)$  of $P\in \overline{\Psi}_0^*(N,E)$
is an element in $\mathcal{G}(S^*N,\pi^*{E})$,
the group of continuous gauge transformations. Set
\begin{align*}
\mathcal{G}_\sharp=\prod_{E\in \textrm{Vect}(N)}\mathcal{G}(S^*N,\pi^*E).
\end{align*}
We also have a well defined
 \textit{symbol map}
\begin{align*}
\sigma_0:\, \textrm{Prin}^\Psi_N(X)\to \textrm{Prin}^{\mathcal{G}_\sharp}_{S^*N}(X)
\end{align*}
in the obvious notation as follows:
For $\{\phi^P_{\alpha\beta}:U_{\alpha\beta}\to \overline{\Psi}_0^*(N,E)\}$ the transition maps of $\mathcal{P}
\in \textrm{Prin}^\Psi_N(X)$, define $\sigma_0(\mathcal P)\in \textrm{Prin}^{\mathcal{G}_\sharp}_{S^*N}(X)$  by the transition maps
$$\sigma_0(\phi^\mathcal{P}_{\alpha\beta}):U_{\alpha\beta}\to \mathcal{G}(S^*N,\pi^*{E}).$$
$\sigma_0$ is continuous precisely because we are using the uniform topology on $\mathcal{G}_\sharp$.

The map $\sigma_0$ can be defined similarly
on vector bundles. Since $\Gamma^s(N,E)$ and $\Gamma^s(S^*N,\pi^*E)$ are faithful
 representations of $\Psi^*_0$ and $\mathcal{G}_\sharp$, respectively, there is no loss of information by considering
\begin{align*}
\sigma_0:\, \textrm{Vect}^\Psi_N(X)\to \textrm{Vect}^{\mathcal{G}_\sharp}_{S^*N}(X).
\end{align*}
Note that if the fibers of $\mathcal{E}\in {\rm Vect}^\Psi_N(X)$ are
isomorphic to $\Gamma^s(N,E)$, then
$\sigma_0(\mathcal{E})$ has fibers isomorphic to $\Gamma^s(S^*N,\pi^*E)$.

From the definition of sums and tensor products of $\Psi$-vector bundles, $\sigma_0$ induces a ring homomorphism
\begin{align*}
\sigma_0:K^\Psi_N(X)\to K^{\mathcal{G}_\sharp}_{S^*N}(X).
\end{align*}
The leading order Chern character $
ch^\Psi:K^\Psi_N(X)\to H^{\textrm{ev}}(X)$ is defined to be the composition
$$ch^\Psi= ch^{\mathcal{G}_\sharp}\circ\sigma _0.$$
For $ch^\Psi$ on $K^\Psi(X)$, we extend $\sigma _0$ by
$$\sigma_0(a_1\otimes\ldots\otimes a_k)=\sigma_0(a_1)\otimes\ldots
\otimes\sigma_0(a_k),$$
and applying $ch^{\mathcal{G}_\sharp}$ as before.

\subsection{\textbf{Nontrivial Elements in $K^\Psi$}}
We can detect nonzero elements of $K^\psi_N(X)$ using
the symbol map, the leading order Chern character and the ordinary Chern character.
This extends techniques in \cite{ALH1} which produced nontrivial examples
of bundles in Vect${}^\Psi_N(X).$

An element of $\mathcal{G}(N,E)$ is an invertible pseudodifferential operator of order zero acting on $\Gamma^s(N,E)$, so we have a canonical inclusion
\begin{align*}
j: \textrm{Vect}^\mathcal{G}_N(X)\hookrightarrow \textrm{Vect}^\Psi_N(X).
\end{align*}
The projection $\pi: S^*N\to N$ induces a pullback map $$\pi^*:\mathcal{G}(N,E)\to\mathcal{G}(S^*N,\pi^*E).$$
Taking pullbacks of transition maps, this induces
$
\pi^*:\textrm{Vect}^\mathcal{G}_N(X)\rightarrow \textrm{Vect}^{\mathcal{G}_\sharp}_{S^*N}(X)
$ and
$$\pi^*:K^\mathcal{G}_N(X)\rightarrow K^{\mathcal{G}_\sharp}_{S^*N}(X).$$
The following diagram commutes:
$$\label{giantdiag1}
\begin{diagram}
\node{K^\Psi_N(X)}\arrow{e,t}{\sigma_0}\node{K^{\mathcal{G}_\sharp}_{S^*N}(X)}\arrow{e,t}{
ch^{\mathcal{G}_\sharp}}
\node{H^*(X)}\\
\node{}\node{K_N^\mathcal{G}(X)}\arrow{nw,r}{j}\arrow{n,r}{\pi^*}\arrow{e,b}{
ch^\mathcal{G}}
\node{H^*(X)}\arrow{n,r}{\cdot {\rm vol}(S^*N)}
\end{diagram}
$$
The  vertical map on the right is multiplication by the constant $\textrm{vol}(S^*N)$.
The diagram shows that for  $\vartheta\in K^\mathcal{G}_N(X)$ with
 nonvanishing leading order Chern character,
  $\pi^*(\vartheta)\in K^{\mathcal{G}_\sharp}_{S^*N}(X)$
   and $j(\vartheta)\in K^\Psi_N(X)$ are nontrivial.
For example,  replacing $S^1$ in \S\ref{exloop} by $N$, we see that
if $CH(T_\mathbb{C}X)\neq 0$ then  $[\mathcal{F}]\in K^\mathcal{G}_N(X)$ is nonzero.
Thus  $CH(T_\mathbb{C}X)\neq 0$ implies $\pi^*(\mathcal{F})\in K_{S^*N}^{\mathcal{G}_\sharp}(X)$ and $j(\mathcal{F})\in K^\Psi_N(X)$ are nontrivial.

We have  $K^{\mathcal{G}_\sharp}_{S^*N}(X)\cong K(X\times S^*N)$ and $K^{\mathcal{G}}_N(X)\cong K(X\times N)$. By  the  K\"unneth formula, $H^*(X\times S^*N)\cong H^*(X)\otimes H^*(S^*N)$  and
$H^*(X\times N)\cong H^*(X)\otimes H^*(N)$. The corresponding commutative diagram for
$CH$ is
\begin{align}
\begin{diagram}\label{giantdiag2}
\node{K^\Psi_N(X)\otimes\mathbb{C}}\arrow{e,t}{\sigma_0}\node{K(X\times S^*N)\otimes\mathbb{C}}\arrow{e,tb}{
CH}{\cong}
\node{H^*(X)\otimes H^*(S^*N)}\\
\node{}\node{K(X\times N)\otimes\mathbb{C}}\arrow{nw,r}{j}\arrow{n,r}{\pi^*}\arrow{e,tb}{
CH}{\cong}
\node{H^*(X)\otimes H^*(N)}\arrow{n,r}{1\otimes \pi^*}
\end{diagram}
\end{align}
Let dim$(N)=2l=n$ and let $[w_N]$ a generator of $H^{2l}(N)$. Set
$$\textit{T}=CH^{-1}\left(H^*(X)\otimes\left(\bigoplus_{i=0}^{\textrm{dim}N-1}H^i(N)\right)\right).$$
In other words, $\textit{T}$ is the subset of $K^\mathcal{G}_N(X)\otimes \mathbb{C}$ consisting of $\gamma$
such that $CH(\gamma)$ does not contain a term of the form $a\otimes [w_N]$ for $a\in H^*(X)$.

 We now prove that  $K^\mathcal{G_\sharp}_{S^*N}(X)$ and
 $K^\Psi_N(X)\otimes \mathbb C$ is at least as large as $T$.

 \begin{theo}\label{lastthm}

\begin{enumerate}[a)]
\item $\pi^*:T\to K^\mathcal{G_\sharp}_{S^*N}(X)\otimes\mathbb{C}$ and $j:T\to K_N^\Psi(X)\otimes\mathbb{C}$ are injections.
 \item If $\chi(N)=0$, then $\pi^*:K^{\mathcal{G}}_{N}(X)\otimes\mathbb{C}\to K^{\mathcal{G}_\sharp}_{S^*N}(X)\otimes\mathbb{C}$ and $j:K^{\mathcal{G}}_{N}(X)\otimes\mathbb{C}\to K^{\Psi}_{N}(X)\otimes\mathbb{C}$ are injections.
\end{enumerate}
\end{theo}
\begin{proof}
(a) The Gysin sequence gives
\begin{align*}
\ldots\to H^{2k-n}(N)\stackrel{\cup e}{\to} H^{2k}(N)
\stackrel{\pi^*}{\to} H^{2k}(S^*N)\to H^{2k-n+1}(N)\ldots,
\end{align*}
where $e$ is the Euler class of $S^*N$.
In particular, for $k<l$, $\pi^*:H^{2k}(N)
\to H^{2k}(S^*N)$ is an injection.

If $CH(\gamma)$ does not contain a term of the form
$a\otimes [w_N] $, then $(1\otimes \pi^*)(CH(\gamma))
\in H^*(X)\otimes H^*(S^*N)$ is nonzero.  By the commutativity of
(\ref{giantdiag2}),
   $\pi^*(\gamma)\in K(X\times S^*N)\otimes\mathbb{C}=K^\mathcal{G_\sharp}_{S^*N}(X)\otimes\mathbb{C}$ and $j(\gamma)\in K^\Psi_N(X)\otimes\mathbb{C}$ are nonzero.

   (b)
If $\chi(N)=0$ then $\pi^*:H^{2l}(N)
\to H^{2l}(S^*N)$ is an injection. In this case, $1\otimes \pi^*$ is injective and the argument in (a) extends from $T$ to all of $K^{\mathcal{G}}_{N}(X)\otimes \mathbb C$.
\end{proof}

Since $CH$ is an isomorphism, the nontrivial elements in $K^\Psi_N(X)\otimes\mathbb{C}$
given in Theorem \ref{lastthm} can be identified with most  of
$H^*(X\times N).$
\begin{coro}\label{finalcoro}
A copy of $H^*(X)\otimes\left(\bigoplus_{i=0}^{\textrm{dim} \ N-1}H^i(N)\right)$ injects into $K^\Psi_N(X)\otimes\mathbb{C}$.
\end{coro}

\bibliography{bibliopgd}

\end{document}